\documentclass[reqno]{amsart}
\usepackage[margin=3cm]{geometry}
\usepackage{amsthm}
\usepackage{comment}
\usepackage{amsmath}
\usepackage{mathtools}
\usepackage{amsfonts}
\usepackage{amssymb}
\usepackage{xcolor}
\usepackage[dvipsnames]{xcolor}
\newtheorem{theorem}{Theorem}[section]
\newtheorem*{hopf}{Hopf Decomposition Theorem}

\newtheorem{proposition}[theorem]{Proposition}

\newtheorem{claim}[theorem]{Claim}

\newtheorem{corollary}[theorem]{Corollary}
\newtheorem{theoremletter}{Theorem}

\theoremstyle{definition}
\newtheorem{question}{Question}

\newtheorem*{rmk}{Remark}

\usepackage{hyperref}
\usepackage[capitalize,nameinlink]{cleveref}
\hypersetup{
    colorlinks,
    citecolor=red,
    filecolor=black,
    linkcolor=cyan,
    urlcolor=black
}

\crefname{question}{Question}{Questions}

\title{Dynamics of composition operators induced by odometers}

\author{Udayan B. Darji}
\address{Department of Mathematics, University of Louisville, Louisville, Kentucky 40292}
\email{ubdarj01@gmail.com}

\author{Daniel Gomes}
\address{Departamento de Matemática, Instituto de Matemática, Estatística e Computação Científica, Universidade Estadual de Campinas, Rua Sérgio Buarque de Holanda, 651, 13083-970, Campinas-SP, Brazil}
\email{danielgomes@ime.unicamp.br}

\author{Régis Varão}
\address{Departamento de Matemática, Instituto de Matemática, Estatística e Computação Científica, Universidade Estadual de Campinas, Rua Sérgio Buarque de Holanda, 651, 13083-970, Campinas-SP, Brazil}
\email{varao@unicamp.br}
\date{}


\thanks{The second author was supported by the São Paulo Research Foundation (FAPESP), grant 2021/02672-2. The third author was supported by the São Paulo Research Foundation (FAPESP), grant 24/15612-6, and  by Conselho Nacional de Desenvolvimento Científico e Tecnológico (CNPq), grant 314978/2023-2. }

\begin{document}

\begin{abstract} 
   We study the linear dynamics of composition operators induced by measurable transformations on finite measure spaces, with particular emphasis on operators induced by odometers. Our first main result shows that, on a finite measure space, supercyclicity of a composition operator implies hypercyclicity. This phenomenon has no analogue in several classical settings and highlights a rigidity specific to the finite-measure context.

We then focus on composition operators induced by odometers and show that many dynamical properties that are distinct for weighted backward shifts collapse in this setting. In particular, for such operators, supercyclicity, Li–Yorke chaos, hypercyclicity, weak mixing, and Devaney chaos are all equivalent.

In contrast to this collapse, we show that the classical equivalence between Devaney chaos and the Frequent Hypercyclicity Criterion for weighted backward shifts fails for odometers. Specifically, we construct a mixing, chaotic, and distributionally chaotic composition operator that does not satisfy the Frequent Hypercyclicity Criterion. This combination of rigidity and separation demonstrates that the dynamical behavior of composition operators induced by odometers differs sharply from that of weighted backward shifts.
\end{abstract}

\maketitle

\section{Introduction}

Linear dynamics studies the long-term behavior of iterates of continuous linear operators on Banach or Fréchet spaces, and lies at the intersection of operator theory and topological dynamics. A central theme in the subject is the analysis of various notions of recurrent or chaotic behavior, which form a natural hierarchy of dynamical complexity: supercyclicity, Li-Yorke chaos, hypercyclicity, mixing, Devaney chaos, distributional chaos, and frequent hypercyclicity. Understanding how these properties relate to one another, whether they coincide, whether one implies another, or whether they can be completely separated, is one of the driving questions in the field.

Weighted backward shifts on $\ell^p$-spaces form a fundamental and well-understood class of examples. They provide explicit models separating many dynamical properties, but they do not distinguish Devaney chaos from frequent hypercyclicity and satisfying the Frequent Hypercyclicity Criterion.

Weighted shifts arise as a special case of composition operators on $L^p$-spaces. If $(X,\mathcal{B},\mu)$ is a $\sigma$-finite measure space and $f:X\to X$ is non-singular, the associated composition operator $T_f(\varphi)=\varphi\circ f$ defines a broad class of operators whose dynamics reflect the underlying measurable system. The operator-theoretic study of such maps was first undertaken in linear dynamics in \cite{bayart2018topological} and has since developed in numerous works  (see, for instance, \cite{bongiorno2022linear,daniello2021generalized,D’Aniello2025supercyclic,darji2021,gomes2024}). Weighted backward shifts arise precisely when the system is dissipative and exhibits bounded distortion, a setting in which the dynamics of $T_f$ mirror those of classical shifts \cite{daniello2022shiftlike,D’Aniello2025supercyclic,daniello2024interplay}.

By the Hopf Decomposition Theorem, every non-singular system splits into a dissipative and a conservative part. While the dissipative case is largely captured by weighted shifts, the conservative case is less explored. Among conservative systems, non-singular odometers form the simplest nontrivial examples. 

In this paper we study the linear dynamics of composition operators induced by odometers. We show that their behavior differs markedly from that of weighted shifts. In particular, we construct mixing and chaotic composition operators induced by odometers that do not satisfy the Frequent Hypercyclicity Criterion. Conversely, several dynamical notions collapse in this setting: for composition operators induced by odometers, supercyclicity, Li--Yorke chaos, hypercyclicity, weak mixing, and Devaney chaos are all equivalent.

\section{Background and main results}
We let $\mathbb{N}$ denote the set of positive integers, i.e., $\mathbb{N}= \{1, 2, \ldots\}$.

Let $\mathcal{X}$ be a separable Banach space over $\mathbb{K}=\mathbb{R}$ or $\mathbb{C}$, and let $T:\mathcal{X}\to \mathcal{X}$ be a continuous linear operator. Several notions of chaotic behavior have been studied in linear dynamics. We recall some of those relevant for the present work. We say that $T$ is
\begin{itemize}
    \item \textit{supercyclic} if there exists $x\in\mathcal{X}$ such that $\{\lambda T^n x : \lambda\in \mathbb{K},\, n\ge 1\}$ is dense in $\mathcal{X}$;
    \item \textit{hypercyclic} if there exists $x\in\mathcal{X}$ such that $\{T^n x : n\ge 1\}$ is dense in $\mathcal{X}$;
    \item \textit{Devaney chaotic} if $T$ is hypercyclic and admits a dense subset of periodic points;
    \item \textit{mixing} if for every pair of nonempty open sets $U,V\subseteq\mathcal{X}$ there exists $n_0\in\mathbb{N}$ such that $T^n(U)\cap V\neq\varnothing$ for all $n\ge n_0$;
    \item \textit{frequently hypercyclic} if there exists $x\in \mathcal{X}$ such that for every non-empty open set $U\subseteq \mathcal{X}$
    \[
    \liminf_{N\to\infty}
    \frac{1}{N}\#\{\, 1\leq n \leq N : T^{n}x \in U \,\}  > 0;
    \]
    \item \textit{Li--Yorke chaotic} if there exists $x\in\mathcal{X}$ such that
    \[
        \liminf_{n\to\infty}\|T^n x\| = 0 
        \qquad\text{and}\qquad 
        \limsup_{n\to\infty}\|T^n x\| = \infty;
    \]
    \item \textit{distributionally chaotic} if there exist $x\in\mathcal{X}$ and sets $D,E\subseteq\mathbb{N}$ with $\overline{\mathrm{dens}}(D)=\overline{\mathrm{dens}}(E)=1$ such that 
    \[
        \lim_{n\in D} \|T^n x\| = 0
        \qquad\text{and}\qquad
        \lim_{n\in E}\|T^n x\| = \infty.
    \]
\end{itemize}
We note that Li-Yorke chaos and distributional chaos are defined in general metric spaces. The definitions we stated above are characterization theorems in the context of linear dynamics given respectively in \cite{bernardes2020li} and \cite{bernardes2018}.
For general references on linear dynamics we refer to the books \cite{bayart2009dynamics} and \cite{grosse2011linear}.

A powerful tool to show that an operator is mixing, Devaney chaotic, frequently hypercyclic or distributionally chaotic is to show that it satisfies the Frequent Hypercyclicity Criterion (see \cite{bayart2006frequently, BERNARDES20132143, bonilla2007frequently}).

However, the converse is not always true, that is, there exist frequently hypercyclic, Devaney chaotic, mixing and distributionally chaotic operators that do not satisfy the Frequent Hypercyclicity Criterion. Therefore one can ask how ``chaotic" an operator can be without satisfying the Frequent Hypercyclicity Criterion. For weighted backward shifts, we have that the Frequent Hypercyclicity Criterion coincides with frequent hypercyclicity and Devaney chaos (see \cite[Theorem~4.8 and Proposition~9.13]{grosse2011linear}). We will show in Theorem \ref{thm: dist chaos} that in our context there exist operators that are Devaney chaotic, mixing and distributionally chaotic that do not satisfy the Frequent Hypercyclicity Criterion. 

We now introduce the main framework of our study: \textit{composition operators acting on $L^p$, $1\leq p<\infty$.} Let $(X,\mathcal{B},\mu)$ be a $\sigma$-finite measure space and let $f:X\to X$ be a measurable and non-singular transformation (that is, $\mu(B)=0$ implies $\mu(f^{-1}(B))=0$ for all $B\in\mathcal{B}$). For $1\le p<\infty$, we denote by
\[
    L^p(X;\mathbb{K})
    = \Big\{ \varphi:X\to\mathbb{K} : \int_X |\varphi|^p\, d\mu < \infty \Big\}
\]
the usual Banach space of $p$-integrable functions. The composition operator $T_f : L^p(X;\mathbb{K})\to L^p(X;\mathbb{K})$ induced by $f$ is defined by
\[
    T_f(\varphi) = \varphi\circ f .
\]
The non-singularity of $f$ ensures that $T_f$ is well-defined. We additionally assume the existence of a constant $c>0$ such that 
\[
    \mu(B) \le c\, \mu(f^{-1}(B))\quad\text{for all } B\in\mathcal{B},
\]
so that $T_f$ is continuous (see \cite{singh1993composition}).

Supercyclic composition operators have been studied recently in \cite{D’Aniello2025supercyclic}, where supercyclicity is characterized in the context of dissipative systems with bounded distortion. We show that if the measure $\mu$ is finite, then supercyclicity coincides with hypercyclicity even if we do not assume that the system is dissipative.

\begin{theoremletter}\label{thm: sc implies hc}
    Let $(X,\mathcal{B},\mu)$ be a finite measure space and $f:X\to X$ be a measurable non-singular transformation such that $T_f:L^p(X;\mathbb{K})\to L^p(X;\mathbb{K})$ is continuous. If $T_f$ is supercyclic, then it is hypercyclic.
\end{theoremletter}

Measure-theoretic properties of $f$, particularly conservativity and dissipativity, play a key role in understanding the dynamics of $T_f$ (see, for instance, \cite{daniello2022shiftlike, D’Aniello2025supercyclic, darji2021, gomes2024}). We say that the system $(X,\mathcal{B},\mu,f)$ is
\begin{itemize}
    \item \textit{conservative} if for every $B\in\mathcal{B}$ with $\mu(B)>0$ there exists $n\in\mathbb{N}$ such that $\mu(B\cap f^{-n}(B))>0$;
    \item \textit{dissipative} if there exists $W\in\mathcal{B}$ such that the sets $f^n(W)$, $n\in\mathbb{Z}$, are pairwise disjoint and 
    \[
        X = \bigcup_{n\in\mathbb{Z}} f^n(W).
    \]
\end{itemize}
The next theorem, which can be found in \cite[Theorem~3.2]{Krengel1985}, provides the canonical decomposition of a non-singular system into its conservative and dissipative parts. 

\begin{hopf}
    If $(X,\mathcal{B},\mu)$ is a $\sigma$-finite measure space and $f:X\to X$ is a non-singular measurable transformation, then $X$ can be written as a disjoint union of two $f$-invariant sets $\mathcal{C}(f)$ and $\mathcal{D}(f)$ such that $f|_{\mathcal{C}(f)}$ is conservative and $f|_{\mathcal{D}(f)}$ is dissipative.
\end{hopf}

We are particularly interested in a class of conservative composition operators: the ones induced by odometers. Let $\{\alpha_n\}_{n=1}^\infty$ be a sequence of integers such that $\alpha_n\geq 2$ for all $n\geq 1$. For each $n\geq 1$, consider $A_n=\{0,\ldots\,\alpha_n-1\}$ and let $X=\prod_{n=1}^\infty A_n$ endowed with the product topology. We define the odometer $f:X\to X$ by: if $x=x^*:=(\alpha_1-1,\alpha_2-1,\alpha_3-1,\ldots)$, then $f(x)=(0,0,0,\ldots)$. If $x=(x_1,x_2,x_3,\ldots) \neq x^*$, let $l(x)=\min\{j\geq 1:x_j\neq \alpha_j-1\}$ and define $f(x)$ by
\[
f(x)_n=\begin{dcases}
    0 & n<l(x),\\
    x_n+1 & n=l(x),\\
    x_n & n>l(x).
\end{dcases}
\]

We may consider the $A_n$ as groups with the usual operation of mod $\alpha_n$ addition. This yields a natural group operation on $X$, namely point-wise addition with carry over. As such, $X$ is a topological group. 
Note that each $k \in \mathbb{N}$ has a unique representation in $X$, namely $(k(1),k(2),\ldots) \in X$ that $k = \sum_{i=1}^{\infty} k(i) \Pi_{j=1}^{i-1} \alpha_j$. Naturally, $k(i) =0$ for all but finitely many $i$. Hence, for $x \in X$ and $k \in \mathbb{N}$, $f^k(x)$ is simply the sum of $k= (k(1), k(2), \ldots) +x $, under the group operation on $X$. Similarly, for each $k \in \mathbb{N}$, $-k$ has a unique representation in $X$, i.e., there exists $y \in X$ such that $y + (k(1), k(2), \ldots) = \underline{0}$ where $\underline{0}$ is the identity of $X$. In particular, $y = (y_1, y_2, \ldots)$ is given by
\[
y_i =\begin{dcases}
    0 & i < j ,\\
    \alpha_j - k(j) & i = j,\\
     \alpha_i - k(i) -1 & i >j .
\end{dcases}
\]
where $j$ is the least integer where $k(j) \neq 0$. Note that $y(i)$ is $\alpha_i-1$ for all but finitely many $i$. As earlier $f^{-k}(x)$ is simply the sum of $y +x$ in $X$ under the group operation of $X$.

Let $\mathcal{B}$ be the Borel $\sigma$-algebra of $X$. For every $n\geq 1$, let $\mu_n$ be a probability on $A_n$ such that $\mu_n(a)>0$ for all $a\in A_n$ (we use the standard notation $\mu_n(a)$ in place of $\mu_n(\{a\})$) and let $\mu=\prod_{n=1}^\infty \mu_n$ be the product probability on $X$. We will always assume that $\prod_{n=1}^\infty \max \{\mu_n(a):a\in A_n\}=0$, so that the measure $\mu$ is non-atomic. We call $(X,\mu,f)$ an \textit{odometer system}. For a treatment of odometers from the ergodic theory point of view, we refer to \cite{danilenko2023ergodic} and \cite{hawkins2021ergodic}.

If our results regarding the composition operators are valid for both $\mathbb{K}= \mathbb{R}$ and $\mathbb{C}$, we write $L^p(X)$ instead of $L^p(X;\mathbb{K})$. If $f:X\to X$ is an odometer, $T_f:L^p(X)\to L^p(X)$ is continuous if, and only if (see \cite{danilenko2023ergodic}),
\begin{equation}
 \inf \Big\{ \lambda_{n}(j) \prod_{k=0}^{n-1} \lambda_{k}(0) : l \ge 1,\; j \in A_{n} \Big\}>0,
\tag{$\ast$} \label{odoCont}   
\end{equation}

where 
\[
\lambda_n(j)=\frac{\mu_n(j)}{\mu_n(j-1)}
\]
if $j\neq 0$ and
\[
\lambda_n(0)=\frac{\mu_n(0)}{\mu_n(\alpha_n-1)}.
\]

We have the following basic results for all odometer systems.
\begin{theoremletter}\label{thm:periodicNotFHC}
  Suppose that $(X,\mu,f)$ is an odometer system such that $T_f:L^p(X)\to L^p(X)$ is continuous. Then the following statements hold.
  \begin{enumerate}
      \item The set of periodic points of $T_f$ is dense in $L^p(X)$.
      \item If $p \ge 2$, then $T_f$ fails to satisfy the Frequent Hypercyclicity Criterion. 
  \end{enumerate}
\end{theoremletter}
We also have that a variety of chaos-type properties coalesce for odometer systems.
\begin{theoremletter}\label{thm: LY implies hc}
    If $(X,\mu,f)$ is an odometer system such that  $T_f$ is continuous, then the following are equivalent:
    \begin{enumerate}
        \item $T_f$ is Li-Yorke chaotic;
        \item $T_f$ is supercyclic;
        \item $T_f$ is hypercyclic;
        \item $T_f$ is Devaney chaotic.
    \end{enumerate}
\end{theoremletter}

 Moreover, we construct an explicit example of an odometer system $(X,\mu,f)$ which shows that distributional chaos can occur in an odometer system. This example enjoys a variety of other properties but does not satisfy the Frequent Hypercyclicity Criterion. This is markedly different from the weighted backward shifts and dissipative context.

\begin{theoremletter}\label{thm: dist chaos}
    There exists an odometer system $(X,\mu,f)$ such that $T_f:L^p(X)\to L^p(X)$ is mixing, distributionally chaotic and  Devaney chaotic. Moreover, for $p \ge 2$, $T_f$ fails to satisfy the Frequently Hypercyclic Criterion. 
\end{theoremletter}

In the following, we collect some  questions that might be of interest. Throughout, assume that $(X,\mathcal{B},\mu)$ is a $\sigma$-finite measure space and $f:X\to X$ is a measurable non-singular transformation such that $T_f:L^p(X)\to L^p(X)$ is continuous. 

The first question is the analogue of Theorem \ref{thm: sc implies hc} for Li-Yorke chaos.
\begin{question}
    Suppose $\mu$ is finite and $T_f$ is Li-Yorke chaotic. Does this imply $T_f$ is hypercyclic? Such is the case if $f$ is dissipative as this implies that $T_f$ satisfies the Frequent Hypercyclicity Criterion (see \cite[Theorems~3.1 and ~3.3]{darji2021}). More generally, if $(X,\mu,f)$ is a conservative system with $\mu$ finite, does Theorem~\ref{thm: LY implies hc} hold?
\end{question}

\begin{question}
    Is there an odometer system $(X,\mu,f)$ such that $T_f$ is hypercyclic but is not frequently hypercyclic or distributionally chaotic? 
\end{question}

The proofs of the main results are distributed across the subsequent sections.

\section{Proof of Theorem A }\label{sec: sc}
In this section we study supercyclicity for composition operators, giving a series of results that lead to the proof of Theorem \ref{thm: sc implies hc}. Let $\mathcal{X}$ be a Banach space over $\mathbb{K}=\mathbb{R}$ or $\mathbb{C}$ and $T:\mathcal{X}\to \mathcal{X}$ be a continuous linear operator. We say that $T$ is
\begin{itemize}
    \item $\mathbb{R}$-\textit{supercyclic} if there exists $x\in \mathcal{X}$ such that $\{\lambda T^nx: \lambda\in \mathbb{R}, n\geq 1\}$ is dense in $\mathcal{X}$;
    \item $\mathbb{R}_+$-\textit{supercyclic} if there exists $x\in \mathcal{X}$ such that $\{\lambda T^nx: \lambda\in \mathbb{R_+}, n\geq 1\}$ is dense in $\mathcal{X}$;
\end{itemize}
It was shown by Berm\'udez et al. \cite[Theorem 2.1]{bermudez2002} that $\mathbb{R}$- and $\mathbb{R}_+$-supercyclic operators coincide. In the next theorem we show that if $\mu$ is finite, then $T_f$ being supercyclic is equivalent to it being hypercyclic, by applying the characterization of hypercyclic  composition operators given in \cite[Theorem~1.1]{bayart2018topological}.

\begin{theorem}\label{thm: r-sc implies hc}
    Suppose that $\mu(X)<\infty$ and $T_f:L^p(X;\mathbb{K})\to L^p(X;\mathbb{K})$ is $\mathbb{R}$-supercyclic. Then $T_f$ is hypercyclic.
\end{theorem}

\begin{proof}
    By \cite[Theorem 2.1]{bermudez2002}, $T_f$ is also $\mathbb{R}_+$-supercyclic. Hence, there exist $k\geq 1$, $\lambda\in \mathbb{R}_+$ and $\varphi\in L^p(\mu)$ such that
    \[
    \|\varphi-2\|<\frac{\varepsilon}{2} \hspace{0.5cm} \text{and} \hspace{0.5cm} \|\lambda\varphi\circ f^k+2\|<\frac{\varepsilon}{2}.
    \]
    Since
    \begin{align*}
        \|\varphi-2\|^p&=\int_{\{x\in X:|\varphi(x)-2|< 1\}} |\varphi-2|^p d\mu+\int_{\{x\in X:|\varphi(x)-2|\geq 1\}} |\varphi-2|^p d\mu\\
        &\geq \mu(\{x\in X:|\varphi(x)-2|\geq 1\}),
    \end{align*}
    we have
    \begin{equation}\label{eq: measure of sets}
        \mu(\{x\in X:|\varphi(x)-2|\geq 1\})<\frac{\varepsilon}{2}.
    \end{equation}

    Analogously,
    \begin{equation}\label{eq: measure of sets 2}
        \mu(\{x\in X:|\lambda\varphi\circ f^k(x)+2|\geq 1\})<\frac{\varepsilon}{2}.
    \end{equation}

    Now let 
    \[
    C=\{x\in X: |\lambda\varphi(x)+2|<1\} \hspace{0.5cm} \text{and} \hspace{0.5cm} D=\{x\in X: |\varphi(x)-2|<1\},
    \]
    and define
$ B=D\cap f^{-k}(C)$.

    First, note that 
    \[
    X\setminus D\subseteq \{x\in X: |\varphi(x)-2|\geq 1\}
    \] and 
    \[
    X\setminus f^{-k}(C)\subseteq \{x\in X: |\lambda\varphi( f^k(x))+2|\geq 1\},
    \]
    so \eqref{eq: measure of sets} and \eqref{eq: measure of sets 2} imply that $\mu(X\setminus B)<\varepsilon$.

    Next, since $\lambda \in \mathbb{R}_+$,
    \[ 
    f^{-k}(D)\subseteq \{x\in X: |\lambda\varphi (f^k(x))+2|\geq 1\}, 
    \]
    and
    \[
    C\subseteq \{x\in X: |\varphi(x)-2|\geq 1\}.
    \]
    Hence, by \eqref{eq: measure of sets 2}, $\mu(f^{-k}(B))\leq \mu (f^{-k}(D))<\varepsilon$ and since $f^k(B)\subseteq C$, by \eqref{eq: measure of sets} we also have that $\mu^*(f^k(B))<\varepsilon$, where $\mu^*$ denotes the outer measure associated with $\mu$. The result now follows from \cite[Theorem~1.1]{bayart2018topological}.
\end{proof}

\begin{rmk}
    When $f$ is bijective, bi-measurable and $T_f^{-1}$ is continuous, $\mathbb{R}$-supercyclic composition operators are characterized in \cite[Theorem~3.1]{D’Aniello2025supercyclic}, without the assumption that the measure $\mu$ is finite.
\end{rmk}

Next we study the relationship between the complex-valued composition operator $T_f:L^p(X;\mathbb{C})\to L^p(X;\mathbb{C})$ induced by $f$ and  the real-valued composition operator $\widetilde{T}_f:L^p(X;\mathbb{R})\to L^p(X;\mathbb{R})$ induced by $f$.

\begin{proposition}
     Suppose that $T_f:L^p(X;\mathbb{C})\to L^p(X;\mathbb{C})$ is supercyclic with  $\varphi$ being a supercyclic vector for $T_f$ and $\varphi=\varphi_1+i\varphi_2$, $\varphi_1,\varphi_2\in L^p(X;\mathbb{R})$. For every $\varepsilon>0$ and every $\psi\in L^p(X;\mathbb{R})$, there exist $\lambda\in \mathbb{R}$ and $k\in \mathbb{N}$ such that
    \[
    \|\lambda\varphi_1 \circ f^k-\psi\|<\varepsilon \hspace{0.5cm} \text{or} \hspace{0.5cm} \|\lambda\varphi_2 \circ f^k-\psi\|<\varepsilon.
    \]
\end{proposition}

\begin{proof}
    Let $\varepsilon>0$ and $\psi\in L^p(X;\mathbb{R})$ with $\|\psi\|\neq 0$. In particular we have that $\psi\in L^p(X;\mathbb{C})$. Let $\lambda=\lambda_1+i\lambda_2$, $\lambda_1,\lambda_2\in \mathbb{R}$, and $k\in \mathbb{N}$ be such that
    \[
    \|\lambda \varphi \circ f^k-\psi \|<\frac{\varepsilon}{3}.
    \]
    By modifying $\lambda$ slightly, we may assume that neither $\lambda_1$ nor $\lambda_2$ is zero. 
    Note that 
    \begin{equation}\label{eq: sc vector}
    \begin{aligned}
        \lambda \varphi \circ f^k &= (\lambda_1+i\lambda_2)(\varphi_1 \circ f^k+i\varphi_2 \circ f^k)\\
        &=(\lambda_1\varphi_1 \circ f^k-\lambda_2\varphi_2 \circ f^k)+i(\lambda_2\varphi_1 \circ f^k+\lambda_1\varphi_2 \circ f^k).
    \end{aligned}
    \end{equation}
    
    Let us first consider the case that $|\lambda_1|\leq |\lambda_2|$. Take
    \[
    \eta=\lambda_2\varphi_1 \circ f^k+\lambda_1 \varphi_2 \circ f^k \in L^p(X;\mathbb{R}).
    \]
    By (\ref{eq: sc vector}) we have that $\|\eta\|<\varepsilon/3$. 

    Hence we have that
    \begin{align*}
        \frac{\varepsilon}{3}&>\| \lambda \varphi \circ f^k-\psi\|
        =\Big\|\lambda_1\Big(\frac{\eta-\lambda_1\varphi_2 \circ f^k}{\lambda_2}\Big)-\lambda_2\varphi_2\circ f^k-\psi+i\eta\Big\|\\
        &\geq \Big\|\Big(\frac{-\lambda_1^2}{\lambda_2}-\lambda_2\Big)\varphi_2 \circ f^k-\psi\Big\|-\Big|\frac{\lambda_1}{\lambda_2}\Big|\|\eta\|-\|\eta\|.
    \end{align*}
    Thus
    \[
    \Big\|\Big(\frac{-\lambda_1^2}{\lambda_2}-\lambda_2\Big)\varphi_2 \circ f^k-\psi\Big\|<\varepsilon.
    \]
    Therefore we have approximated $\psi$ by a real multiple of $\varphi_2 \circ f^k$. The case where $|\lambda_2|\leq |\lambda_1|$ is analogous by approximating $\psi$ by a real multiple of $\varphi_1\circ f^k$.
\end{proof}

\begin{corollary}\label{cor: supercyclic}
    Suppose that $T_f:L^p(X;\mathbb{C})\to L^p(X;\mathbb{C})$ is supercyclic with $\varphi=\varphi_1+i\varphi_2$, $\varphi_1,\varphi_2\in L^p(X;\mathbb{R})$, being a supercyclic vector for $T_f$. For all $\psi \in L^p(X;\mathbb{R})$ there exist $j\in \{1,2\}$ and sequences $\lambda_n\in \mathbb{R}$ and $k_n\in \mathbb{N}$, $n\in \mathbb{N}$, such that
    \[
    \|\lambda_n\varphi_j \circ f^{k_n}-\psi\|\to 0
    \]
    as $n\to \infty$.
\end{corollary}

\begin{proposition}\label{prop: sc imples R-sc}
    Suppose that $T_f:L^p(X;\mathbb{C})\to L^p(X;\mathbb{C})$ is supercyclic. Then $\widetilde{T}_f:L^p(X;\mathbb{R}) \to L^p(X;\mathbb{R})$ is $\mathbb{R}$-supercyclic.
\end{proposition}

\begin{proof}
    Let $\varphi=\varphi_1+i\varphi_2$, $\varphi_1,\varphi_2\in L^p(X;\mathbb{R})$, be a supercyclic vector for $T_f$. For $j\in \{1,2\}$, let 
    \[
    \begin{aligned}
        L_j=\big\{\psi\in L^p(X;\mathbb{R}):& \text{ there exist } (\lambda_n)_{n=1}^\infty \in \mathbb{R}^\mathbb{N}, (k_n)_{n=1}^\infty \in \mathbb{N}^\mathbb{N}\\ &\text{ such that } \|\lambda_n\varphi_j\circ f^{k_n}-\psi\|\to 0\big\}.
    \end{aligned}
    \]
    Then each $L_j$ is closed in $L^p(X;\mathbb{R})$ and, by Corollary \ref{cor: supercyclic}, $L^p(X;\mathbb{R})=L_1\cup L_2$. Hence, there is $j_0\in \{1,2\}$ such that $L_{j_0}$ contains an open ball. Thus
  $L^p(X;\mathbb{R})=\overline{\{\lambda\varphi_{j_0}\circ f^k:\lambda\in \mathbb{R},k\in \mathbb{N}\}}.$
\end{proof}

We say that an operator $T:\mathcal{X}\to \mathcal{X}$ is \textit{weakly mixing} if $T\times T$ is hypercyclic. For composition operators on $L^p$, they coincide with the hypercyclic ones (see \cite[Theorem~2.11]{gomes2024}).

\begin{proposition}\label{prop: R-wm implies C-hc}
    If $\widetilde{T}_f:L^p(X;\mathbb{R})\to L^p(X;\mathbb{R})$ is weakly mixing, then $T_f:L^p(X;\mathbb{C})\to L^p(X;\mathbb{C})$ hypercyclic.
\end{proposition}

\begin{proof}
    Let $\varphi=(\varphi_1,\varphi_2)\in L^p(X;\mathbb{R})\times L^p(X;\mathbb{R})$, be a hypercyclic vector for $\widetilde{T}_f\times \widetilde{T}_f$. We claim that $\varphi=\varphi_1+i\varphi_2 \in L^p(X;\mathbb{C})$ is a hypercyclic vector for $T_f$. To this end, let $\psi\in L^p(X;\mathbb{C})$ and $\varepsilon>0$. We write $\psi=\psi_1+i\psi_2$, $\psi_1,\psi_2\in L^p(X;\mathbb{R})$. Since $(\varphi_1,\varphi_2)$ is hypercyclic for $\widetilde{T}_f\times \widetilde{T}_f$, there exists $k\geq 1$ such that
    \[
    \|(\widetilde{T}_f^k\varphi_1,\widetilde{T}^k_f\varphi_2)-(\psi_1,\psi_2)\|<\frac{\varepsilon}{2}.
    \]
    Hence we have
    \[
    \|\varphi_1\circ f^k-\psi_1\|<\frac{\varepsilon}{2} \hspace{0.5cm} \text{and} \hspace{0.5cm}\|\varphi_2\circ f^k-\psi_2\|<\frac{\varepsilon}{2}.
    \]
    Thus
    \begin{align*}
        \|T_f^k \varphi-\psi\|&= \|(\varphi_1+i\varphi_2)\circ f^k-\psi_1+i\psi_2||\\
        &\leq \|\varphi_1\circ f^k-\psi_1\|+\|\varphi_2\circ f^k-\psi_2\| <\varepsilon,\\
    \end{align*}
    which concludes the proof of the claim.
\end{proof}

\begin{theorem}
    Suppose that $\mu(X)<\infty$ and $T_f:L^p(X;\mathbb{C})\to L^p(X;\mathbb{C})$ is supercyclic. Then $T_f$ is hypercyclic.
\end{theorem}

\begin{proof}
    By Proposition \ref{prop: sc imples R-sc} we have that $\widetilde{T}_f:L^p(X;\mathbb{R})\to L^p(X;\mathbb{R})$ is $\mathbb{R}$-supercyclic, which is equivalent of saying that $\widetilde{T}_f$ is $\mathbb{R}_+$-supercyclic, by \cite[Theorem~2.1]{bermudez2002}. From Theorem \ref{thm: r-sc implies hc} we get that $\widetilde{T}_f$ is hypercyclic. By \cite[Theorem~2.11]{gomes2024} $\widetilde{T}_f$ is weakly mixing and, from Proposition \ref{prop: R-wm implies C-hc} we conclude that $T_f$ is hypercyclic.
\end{proof}

Altogether we get the following characterization, which proves Theorem \ref{thm: sc implies hc}.

\begin{corollary}
    Suppose that $\mu(X)<\infty$. The following are equivalent:
    \begin{enumerate}
        \item $T_f:L^p(X;\mathbb{C})\to L^p(X;\mathbb{C})$ is supercyclic;
        \item $\widetilde{T}_f:L^p(X;\mathbb{R})\to L^p(X;\mathbb{R})$ is $\mathbb{R}$-supercyclic;
        \item $\widetilde{T}_f:L^p(X;\mathbb{R})\to L^p(X;\mathbb{R})$ is $\mathbb{R}_+$-supercyclic;
        \item $\widetilde{T}_f:L^p(X;\mathbb{R})\to L^p(X;\mathbb{R})$ is hypercyclic;
        \item $\widetilde{T}_f:L^p(X;\mathbb{R})\to L^p(X;\mathbb{R})$ is weakly mixing;
        \item $T_f:L^p(X;\mathbb{C})\to L^p(X;\mathbb{C})$ is hypercyclic;
        \item $T_f:L^p(X;\mathbb{C})\to L^p(X;\mathbb{C})$ is weakly mixing.
    \end{enumerate}
\end{corollary}

\begin{proof}
    The implication $(1)\implies (2)$ follows from Proposition \ref{prop: sc imples R-sc}. The equivalence between (2) and (3) follows from \cite[Theorem 2.1]{bermudez2002}. By Theorem \ref{thm: r-sc implies hc} we have that (2) implies (4) and, by Proposition \ref{prop: R-wm implies C-hc}, we have that (5) implies (6). The equivalences $(4) \iff (5)$ and $(6) \iff (7)$ follow from \cite[Theorem~2.11]{gomes2024}. The implication $(7)\implies (1)$ is trivial.
\end{proof}

\section{Proof of Theorem B}\label{sec: periodicNotFHC chaos}
\begin{proof}[Proof of Theorem \ref{thm:periodicNotFHC}] Part (1) was proved in \cite[Example 4.3]{darji2021}. 

Let us prove (2). If $T_f: L^p(X) \rightarrow L^p(X)$ were to satisfy the Frequent Hypercyclicity Criterion for $p \ge 2$, by \cite[Theorem 3.1]{darji2021} we would have that $T_f$ satisfies Condition (SC) as stated in that article. As $\mu(X) < \infty$, we can apply \cite[Theorem 3.3]{darji2021} to obtain that $(X,\mu,f)$ is a dissipative system. However, this contradicts that $f$ is  conservative.   
\end{proof}

\section{Proof of Theorem C }\label{sec: ly chaos}

    \begin{proof}[Proof of Theorem \ref{thm: LY implies hc}] 
    
        As $T_f$ is Li-Yorke chaotic, $X$ is of finite measure and $f$ is bijective, by \cite[Theorem~1.5]{bernardes2020li} there exists a measurable set $B$ with $\mu(B)>0$ and an increasing sequence $n_k\to +\infty$ such that 
        \[
        \mu(f^{n_k}(B))\to 0.
        \]

        We will use \cite[Theorem~1.1]{bayart2018topological} to show that $T_f$ is hypercyclic. In particular, for every $\varepsilon >0$, we will construct a measurable set $A \subseteq X$ such that $\mu (A) > 1 -\varepsilon$ and $j \in {\mathbb{N}}$ such that $\mu (f^j(A)) < \varepsilon$. To this end, let $\varepsilon>0$. As $X$ has a basis of clopen sets and $\mu$ is a Borel probability measure on $X$, it follows that the Lebesgue density theorem holds, i.e., almost every point of $B$ has density one. Using this fact and $\prod_{n=1}^\infty \mu_n(0)=0$, we may choose $N>0$ so that 
        \[
        \prod_{n=1}^N\mu_n(0)<\frac{\varepsilon}{2}
        \]
        and $b_1, \ldots, b_N$ be such that
        \[
        \frac{\mu([b_1,\ldots,b_N]\cap B)}{\mu([b_1,\ldots,b_N])}>1-\frac{\varepsilon}{2}.
        \]
        Note that $[b_1,\ldots,b_N]\cap B = \{b_1\}\times \ldots \times \{b_N\} \times M $ where 
        $M \subseteq \prod_{n=N+1}^{\infty}\mathbb{Z}_{\alpha_n}$ with 
       $\mu_{[N+1,\infty)}(M)> 1- \frac{\varepsilon}{2}$. Here 
        $\mu_{[N+1,\infty)}$ denotes the product measure $\prod_{n=N+1}^{\infty}\mu_n$.

        Let \[
        \delta=\frac{\varepsilon}{2}\cdot \min\Big\{\mu_1(d_1)\cdots\mu_N(d_n):(d_1,\ldots ,d_N)\in \prod_{n=1}^N\mathbb{Z}_{\alpha_n}\Big\}.
        \] For positive integers $k$ and $m$, we let $k(m)$ denote the $m$-th coordinate of $k$ in its representation in terms of $\{\alpha_n\}_{n=1}^{\infty}$, that is
        \[
        k=\sum_{m=1}^\infty k(m)\beta_m,
        \]
        with $\beta_m=\prod_{n=1}^m \alpha_n$. Let $k >0$ be such that  $\mu(f^k(B))<\delta$
        and there exists $m >N$ such that $ k(m) >0$.
        Since $\{b_1\}\times \ldots \times \{b_N\} \times M\subseteq B$, we also have that $\mu(f^k(\{b_1\}\times \ldots \times \{b_N\} \times M))<\delta$.

        We now write
        %\[
        $k=k_1+k_2$,
        %\]
        where $k_1$ and $k_2$ are such that 
        \begin{align*}
            k_1(m)&=0, \hspace{0.5cm} \forall \; m\geq N+1\\
            k_2(m)&=0, \hspace{0.5cm} \forall \; 1\leq m\leq N.
        \end{align*}
        Note that $f^{k_1}((b_1,\ldots,b_N,0,0,\ldots))$ is the pointwise addition of 
        \[
        (k(1),\ldots ,k(N), 0,0,\ldots) \quad\text{and} \quad (b_1, \ldots ,b_N, 0,0,\ldots).
        \]
        Let $a_1,\ldots,a_N$ be the first $N$ coordinates of this sum and let $c\in \{0,1\}$ be the carry over from $N$ to $N+1$ position in this sum. Note that 
        \[ 
        f^{k_1}(\{b_1\}\times \ldots \times \{b_N\} \times M) = f^{c\cdot \prod_{n=1}^N\alpha_n}(\{a_1\}\times \ldots \times \{a_N\} \times M).
        \] 
        We now have
        \begin{align*}
            f^k(\{b_1\}\times \ldots \times \{b_N\} \times M)&=f^{k_1+k_2}(\{b_1\}\times \ldots \times \{b_N\} \times M)\\
            &=f^{k_2+c\cdot \prod_{n=1}^N\alpha_n}(\{a_1\}\times \ldots \times \{a_N\} \times M).
        \end{align*}
        By the definition of $k_2$, there exists $M'\subseteq \prod_{n=N+1}^\infty \mathbb{Z}_{\alpha_n}$ such that
        \[
        f^{k_2+c\cdot \prod_{n=1}^N\alpha_n}(\{a_1\}\times \ldots \times \{a_N\} \times M) = \{a_1\}\times \ldots \times \{a_N\} \times M'.
        \]
        As $\mu(f^k(\{b_1\}\times \ldots \times \{b_N\} \times M))<\delta$, it follows that $\mu_{[N+1,\infty)}(M')<\varepsilon/2$.

        We split into two cases:

        \medskip
        \noindent\textbf{Case 1}: c=0. Let
        \[
        A=\prod_{n=1}^N \mathbb{Z}_{\alpha_n} \times M
        \]
        and $j=k_2$. 
        We have that $\mu(A)>1-\varepsilon/2$ and 
        \begin{align*}
            f^j(A) &= \bigcup_{(d_1,\ldots,d_N)\in \prod_{n=1}^N\mathbb{Z}_{\alpha_n}}  f^{k_2}(\{d_1\}\times\ldots \times \{d_N\} \times M)\\ 
            &= \bigcup_{(d_1,\ldots,d_N)\in \prod_{n=1}^N\mathbb{Z}_{\alpha_n}}  \{d_1\}\times\ldots \times \{d_N\} \times M',
        \end{align*} 
        implying that 
        \[
        \mu(f^j(A))=\sum_{(d_1,\ldots,d_N)\in \prod_{n=1}^N\mathbb{Z}_{\alpha_n}} \mu_1(d_1)\cdots \mu_N(d_N) \cdot \mu_{[N+1,\infty)} (M') \leq \frac{\varepsilon}{2}.
        \]

        \noindent\textbf{Case 2}: c=1. Let $j\in \mathbb{N}$ be such that, in its representation in terms of $\alpha_n$,
        \[
        j(m)=\begin{dcases}
            \alpha_m-1 & 1\leq m \leq N,\\
            k_2(m) & m\geq N+1.
        \end{dcases}
        \]
        and 
        \[
        A=\prod_{n=1}^N \mathbb{Z}_{\alpha_n}\setminus\{(0,\ldots,0)\}\times M.
        \]
        We have that 
        \[
        \mu(A)=(1-\mu_1(0)\ldots\mu_N(0))\cdot\mu_{[N+1,\infty)}(M)>\Big(1-\frac{\varepsilon}{2}\Big)\cdot \Big(1-\frac{\varepsilon}{2}\Big)>1-\varepsilon.
        \]
        Since 
        \begin{align*}
            f^j(A)&=\bigcup_{(d_1,\ldots,d_N)\neq (0,\ldots,0)} f^j(\{d_1\}\times\ldots \times \{d_N\} \times M)\\
            &=\bigcup_{(d_1,\ldots,d_N)\neq (\alpha_1-1,\ldots,\alpha_N-1)} \{d_1\}\times\ldots \times \{d_N\} \times M',
        \end{align*}
        we also have that
        \[
        \mu(f^j(A))=\sum_{(d_1,\ldots,d_N)\neq (\alpha_1-1,\ldots,\alpha_N-1)} \mu_1(d_1)\cdots \mu_N(d_N) \cdot \mu_{[N+1,\infty)}(M')\leq \frac{\varepsilon}{2}.
        \]

        For every $\varepsilon >0$, we have constructed set $A \subseteq X$ such that $\mu (A) > 1 -\varepsilon$ and $j \in {\mathbb{N}}$ such that $\mu (f^j(A)) < \varepsilon$. Hence, $T_f$ is hypercyclic by \cite[Theorem~1.1]{bayart2018topological}.
    \end{proof}

\section{Proof of Theorem D}\label{sec: dist chaos}
    For every $n\geq 1$, set $\alpha_n =2$ and $A_n=\{0,1\}$, and consider the odometer $f$ acting on the space $X = \prod_{n=1}^{\infty} A_n$. We endow $X$ with the product measure $\mu=\prod_{n=1}^{\infty}\mu_n$, where the measures $\mu_n$ are defined as follows. Let $m_k = 4^{k}$, $k\ge 1$. 
    For each $n\in \mathbb{N}$, define 
    \[
        p_n=
        \begin{dcases}
            \frac{1}{2^n}, &  \text{ if } n \notin \{m_1, m_2,\ldots\} \,\\
            \frac{1}{2^k}, & \text{ if } n= m_k,
        \end{dcases}
    \]
    and let $\mu_n(1)=p_n$ and $\mu_n(0)=1-p_n$.It is straightforward to verify that the family $\{\mu_n\}$ satisfies Condition (\ref{odoCont}), and therefore $T_f$ is a bounded linear operator. 

    By \cite[Theorem~3.1]{bongiorno2022linear}, the operator $T_f$ is mixing. Moreover, by Theorem~\ref{thm:periodicNotFHC} part (1), $T_f$ is chaotic, while by Theorem~\ref{thm:periodicNotFHC} part (2), it fails to satisfy the Frequent Hypercyclicity Criterion for $p\geq 2$.

    It remains to show that $T_f$ is distributionally chaotic, which we accomplish by applying \cite[Theorem~32]{Bernardes2026}. Specifically, we construct a set $I \subseteq \mathbb{N}$ with upper density one, measurable sets $B_k \subseteq X$, $k\geq 1$, such that, for every $k\geq 1$,
    \begin{equation}
      \lim _{n \rightarrow \infty, \\ n \in I} \mu(f^{-n}(B_k)) = 0 \ \ \ \  \& \ \ \ \lim_{k \rightarrow \infty} \mu(B_k) =0, \tag{DC1} \label{DC 1}
    \end{equation}
    and, for sufficiently large $k$,  
      \begin{equation}
          \frac{1}{2^{m_k-1}} \left | \Big\{1\leq n\leq2^{m_k-1}: \mu(f^{-n}(B_k))>\frac{1}{16}\Big\} \right |>\frac{1}{16} \tag{DC2}\label{DC 2}.
      \end{equation}            

We now define the relevant sets. For each $k \in \mathbb{N}$, let  
    \[
    \begin{aligned}
        B_k:=  \big\{x \in X: \; & x(m_k) = 1  \ \& \  x(n) =0 \text{ for all } n >m_k\\ 
        &\text{ with } n \notin \{m_k, m_{k+1} , m_{k+2}, \ldots \} \big\}.
    \end{aligned}
    \]
    Recalling that $(n(1), n(2), \ldots)$ is the binary representation of $n$, for each $k\geq 1$ let
\begin{gather*}
    I_k :=\{n\in \mathbb{N}: n(i)=0  \text{ for }  i > m_{k+1}-3, \   \\
    \text{ and there exists } i\in \{m_k+3, m_k+3, \dots , m_{k+1}-3 \} \text{ with } n(i)=1\},\end{gather*}
and set    
    \[
        I:=\bigcup_{k \ge 1} I_k.
    \] 
The proof of Theorem~\ref{thm: dist chaos} follows from the next four claims.

\begin{claim}\label{claim:estIk}
    For each $k\geq 1$, we have $|I_k| = 2^{m_{k+1}-3} - 2^{m_{k}+2}.$   
\end{claim}

\begin{proof}
    Note that for $n \in \mathbb{N}$ and $j\geq 1$, we have that $1 \le n  \le 2^j-1$ if and only if $n(i) = 0$ for all $i >j$. Using this and the definition of $I_k$, it follows that
    \begin{align*}
        |I_k| & =  (2^{m_{k+1}-3}-1) - (2^{m_{k}+2}-1)= 2^{m_{k+1}-3} - 2^{m_{k}+2}. 
        %& = 2^{m_{k+1}-3} - 2^{m_{k}+2}.
    \end{align*}
\end{proof}
        
\begin{claim}\label{claim:density}
        The set $I$ has upper density 1.
\end{claim}
   
\begin{proof} This follows directly from Claim~\ref{claim:estIk}.
        Indeed, for each $k\geq 1$,
        \begin{align*}
            \frac{ \left |\{n\in \mathbb{N}:1\leq n\leq 2^{m_{k+1}-3}  \}\cap I \right |}{2^{m_{k+1}-3}} & \ge \frac{ \left |\{n\in \mathbb{N}:1\leq n\leq 2^{m_{k+1}-3}  \}\cap I_k \right |}{2^{m_{k+1}-3}}\\ 
            & = \frac{ |I_k| }{2^{m_{k+1}-3}}= \frac{2^{m_{k+1}-3} - (2^{m_{k}+2})}{2^{m_{k+1}-3}}\\
            %& = \frac{2^{m_{k+1}-3} - (2^{m_{k}+2})}{2^{m_{k+1}-3}}\\
            & = 1 -   2^{ m_k-m_{k+1}+5}
        \end{align*}
        Since $(m_k - m_{k+1}) \rightarrow -\infty$ as $k \rightarrow \infty$, we have that $I$ has upper density 1.   
    \end{proof}

    We will use the following observation concerning binary representation of $-n$ for $n \in \mathbb{N}$. Let $j$ be the smallest index such that $n(j)=1$. Then the binary expansion of $-n$ is obtained by setting all coordinates before $j$ equal to $0$, placing a $1$ at position $j$, and flipping all the subsequent coordinates of $n$. More precisely,
    \[
    -n(i) =\begin{dcases}
    0 & i < j ,\\
    1 & i = j,\\
     \overline{n(i)} & i >j,
\end{dcases}
\] 
where $\overline{n(i)}$ interchanges $0$ and $1$.

    \begin{claim}\label{claim:verifyDC1}
       The sets $\{B_k\}$ satisfy Condition~\emph{\ref{DC 1}}. 
    \end{claim}
    \begin{proof} First, observe that $\mu(B_k) < 1/k$ for every $k\in\mathbb{N}$, and therefore 
    \[
    \lim_{k \rightarrow \infty} \mu(B_k) =0.
    \]

    We now verify the remaining part of Condition~\ref{DC 1}.  Fix $k \in \mathbb{N}$ and let $n\in I_l$ with $l>k$. By the definition of $I_l$, there exists an index $m_l+3\leq i\leq m_{l+1}-3$ such that $n(i)=1$ and $n(j)=0$ for all $j>i$, i.e. $i$ is the biggest index of $n$ in the binary representation with value 1. 
    
    Since $n(i)=1$ and all higher coordinates vanish, we have $-n(j)=1$ for all $j > i$. Moreover, either $-n(i) =0$ or $-n(i-1)=0$, depending on the first coordinate of binary representation of $n$ which is 1.  
    Consequently, there exists at least one index $m_l+2\leq i' \leq m_{l+1}-3$ such that $-n(i')=0$. We choose $i'$ to be the largest such index. Then, $-n(i')=0$ and $-n(i'+1) =1$, and clearly  $m_l < i' < i'+1 < m_{l+1}$. Since $l >k$, for all $x \in B_k$ we have that $x(i') = x(i'+1) =0$. Hence, $f^{-n} (x)(i'+1)=1$ as there is no  carry over to the ($i'+1$)-th coordinate when computing the sum $x+(-n)$ in its binary representation. Thus we have shown that $f^{-n}(B_k) < 1/l$ for all $l > k$ and $n \in I_l$, implying that $\lim _{n \rightarrow \infty, \\ n \in I} \mu(f^{-n}(B_k)) = 0$, and completing the verification of \ref{DC 1}.
    \end{proof}
    
    \begin{claim}\label{claim:dc2}
        Sets $\{B_k\}$ satisfy Condition\emph{~\ref{DC 2}}.
    \end{claim}
    \begin{proof}
        Fix $k>1$ sufficiently large so that $\mu_{m_k-1}(0)$, $\mu_{m_k}(0)$ and $\mu_{[m_k+1,\infty)}(T_k)$ are all greater than $1/2$, where $T_k$ denotes the projection of $B_k$ onto coordinates  $[m_k+1, \infty)$. This is possible since $\mu_n(0) \rightarrow 1$ as $n \rightarrow \infty$, and $ \prod_{n \notin \{m_1, m_2\ldots,\}} \mu_n(0) > 0$. 
        
        For each $1 \le n \le 2^{m_k-2} -1$ and $ c \in \{0,1\}$ we consider the following set
\[
\begin{aligned}
D(n,c) := \Big\{ x \in \prod_{i=1}^{m_k-2} A_i :&
(x(1),\ldots,x(m_k-2)) + (-n(1),\ldots,-n(m_k-2)) \\
&\text{has carry over } c \text{ at the coordinate  } m_k-1 \Big\}.
\end{aligned}
\]
Note that for each $n$, $D(n,0) \cup D(n,1) = \prod_{i=1}^{m_k-2}A_i$ and hence there exists $c(n) \in \{0,1\}$ such that 
\[
\mu_{[1,m_k-2]}(D(n,c(n))) \ge 1/2.
\] 
For $i \in \{0,1\}$, let
\[
F_i:= \{1 \le n \le 2^{m_k-2} -1: c(n) =i \}.
\]
Since $F_0 \cup F_1 = [1, 2^{m_k-2} -1]$, either 
\[
\frac{|F_0|}{2^{m_k-2}-1}\ge \frac{1}{2}
\quad\text{or}\quad
\frac{|F_1|}{2^{m_k-2}-1}\ge \frac{1}{2}.
\]
We will show that there is a set $J\subseteq[1,2^{m_k}-1]$ with
\[
\frac{|J|}{2^{m_k}-1}\ge \frac{1}{16},
\] 
such that $\mu(f^{-n}(B_k)) \ge 1/16$ for all $n \in J$. This verifies \ref{DC 2}. We consider two cases.

\medskip
\noindent\textbf{Case 1}:  $|F_0|/(2^{m_k-2} -1) \ge 1/2. $ 

Let $J$ be the set of all $n \in [1, 2^{m_k}-1]$ such that 
\[
(-n(1),\ldots,-n(m_k-2))\in F_0,\quad -n(m_k-1)=0\quad \text{and}\quad-n(m_k)=1.
\]
Note that $|J| = |F_0|$. Hence, 
\begin{align*}
    \frac{|J|}{2^{m_k} -1 } & \ge   \frac{\frac{1}{2} \cdot (2^{m_k-2} -1)}{2^{m_k} -1 } >  \frac{\frac{1}{2} \cdot (2^{m_k-3}) }{2^{m_k}  } = \frac{1}{16}. \\
    %& >  \frac{\frac{1}{2} \cdot (2^{m_k-3}) }{2^{m_k}  } = \frac{1}{16}. \\
    %& = \frac{1}{16}.
\end{align*}

Fix $n\in J$ and let $x\in B_k$ with $x(m_k-1)=0$. Denote by $y:=f^{-n}(x)=x+(-n)$. Then
\begin{align*}
    f^{-n}(x)&=x+(-n)\\
    &=(x(1),\dots,x(m_k-2), 0,1,x(m_k+1), x(m_k+2), \ldots )\\
    &+(-n(1),\dots,-n(m_k-2), 0,1,1, \ldots)\\
    &=(y(1),\dots,y(m_k-2),0,0,x(m_k+1), x(m_k+2), \ldots ).
\end{align*}
Note that $(y(1),\dots,y(m_k-2)) \in D(n,c(n))$. Using the facts that $c(n)=0$ (i.e. the carry over at step $m_k-1$ is $0$), $x(m_k-1) =0$ and $-n(m_k-1)=0$, we obtain that $y(m_k-1) =0$. The fact that $x(m_k) =1$ and $-n(m_k)=1$ gives us that $y(m_k)=0$ and a carry over of $1$ at step $m_k+1$. This implies that $y(i)=x(i)$ for all $i \ge m_k+1$.

Now we have 
\begin{align*}
    f^{-n}(B_k)&\supseteq \{x+(-n): x\in B_k \text{ and } x+(-n) \text{ has carry over } 0 \text{ at step } (m_k-1)   \}\\
    &\supseteq D(n,0)\times \{0\}\times \{0\} \times T_k.
\end{align*}
Since $\mu$ is the product measure, we have
\begin{align*}
\mu(f^{-n}(B_k))
&\ge \mu_{[1,m_k-2]}(D(n,0)) \cdot \mu_{m_k-1}(0) \cdot \mu_{m_k}(0) \cdot \mu_{[m_k+1,\infty)}(T_k) \ge \frac{1}{16}.
%&\ge \frac{1}{16}.
\end{align*}
Thus $\mu(f^{-n}(B_k)) \geq 1/16$ for all $n \in J$.

\medskip
\noindent\textbf{Case 2:} $ |F_1|/(2^{m_k-2}-1)\ge 1/2$.

Let $J$ be the set of all $n\in[1,2^{m_k}-1]$ such that
\[
(-n(1),\ldots,-n(m_k-2))\in F_1,\quad -n(m_k-1)=1 \quad \text{and} \quad -n(m_k)=0.
\]
The estimate as in Case~1 yields that
\[
\frac{|J|}{2^{m_k}-1}\ge \frac{1}{16}.
\]
Now we proceed, as in the case above,  to show that $\mu(f^{-n}(B_k)) \geq 1/16$ for all $n \in J$.

Fix $n\in J$ and let $x\in B_k$ with $x(m_k-1)=0$. Denote by $y:=f^{-n}(x)=x+(-n)$. Then

\begin{align*}
    f^{-n}(x)&=x+(-n)=(x(1),\dots,x(m_k-2), 0,1,x(m_k+1), x(m_k+2), \ldots )\\
    &+(-n(1),\dots,-n(m_k-2), 1,0,1, \ldots)\\
    &=(y(1),\dots,y(m_k-2),0,0,x(m_k+1), x(m_k+2), \ldots ).
\end{align*}
Note that $(y(1),\dots,y(m_k-2)) \in D(n,c(n))$. Using the facts that $c(n)=1$ (i.e. the carry over at step $m_k-1$ is $1$), $x(m_k-1) =0$ and $-n(m_k-1)=1$, to obtain that $y(m_k-1) =0$. This gives us a carry over of $1$ at step $m_k$. This fact and using that $x(m_k) =1$ and $-n(m_k)=0$, we obtain that $y(m_k)=0$ and a carry over of $1$ at step $m_k+1$, yielding $y(i)=x(i)$ for all $i \ge m_k+1$.

Now we have 
\begin{align*}
    f^{-n}(B_k)&\supseteq \{x+(-n): x\in B_k \text{ and } x+(-n) \text{ has carry over } 1 \text{ at step } (m_k-1)  \}\\
    &\supseteq D(n,1)\times \{0\}\times \{0\} \times T_k.
\end{align*}
Since $\mu$ is the product measure, we have
\begin{align*}
\mu(f^{-n}(B_k))
&\ge \mu_{[1,m_k-2]}(D(n,0)) \cdot \mu_{m_k-1}(0) \cdot \mu_{m_k}(0) \cdot \mu_{[m_k+1,\infty)}(T_k) \ge \frac{1}{16}.
%&\ge \frac{1}{16}.
\end{align*}
Thus $\mu(f^{-n}(B_k)) \geq 1/16$ for all $n \in J$.
\end{proof}

\section*{Acknowledgement}  
    The authors would like to thank Nilson Bernardes and Alfred Peris for some valuable conversations concerning some of the work in this article.
    
\bibliographystyle{acm}
\bibliography{bibliografia}
\end{document}